\documentclass[reqno]{amsart}

\usepackage{amssymb}
\usepackage[all]{xy}

\swapnumbers
\theoremstyle{plain}
\newtheorem{thm}{Theorem}[section]
\newtheorem{lem}[thm]{Lemma}
\newtheorem{cor}[thm]{Corollary}
\newtheorem{prop}[thm]{Proposition}

\theoremstyle{definition}
\newtheorem{ntt}[thm]{}
\newtheorem{ex}[thm]{Example}
\newtheorem{rem}[thm]{Remark}
\newtheorem{dfn}[thm]{Definition}

\newcommand{\Gm}{\mathbb{G}_{m}}  % multiplicative group scheme
\newcommand{\zz}{\mathbb{Z}}       % integers
\newcommand{\N}{\mathbb{N}}        % natural numbers
\newcommand{\A}{\mathbb{A}}        % affine space
\newcommand{\cc}{\mathfrak{c}}     % characteristic map
\newcommand{\too}{\,\longrightarrow\,}  % long right arrow
\newcommand{\LU}{\mathcal{U}}     % localization scheme
\DeclareMathOperator{\Cob}{\Omega}   % Algebraic Cobordism
\newcommand{\smb}{{\scriptscriptstyle \bullet}} % ``Small'' point
\DeclareMathOperator{\Spec}{\mathrm{Spec}}  % Spectrum
\DeclareMathOperator{\Rep}{\mathrm{Rep}}    % Representation ring

\DeclareMathOperator{\cone}{\mathrm{cone}}
\DeclareMathOperator{\cfb}{\mathfrak{cf}}
\newcommand{\sheaf}[1]{{\mathcal #1}}

\DeclareMathOperator{\Nis}{\mathrm{Nis}}
\DeclareMathOperator{\pt}{\mathrm{pt}}   % point
\DeclareMathOperator{\GL}{\mathrm{GL}}              % General linear group
\DeclareMathOperator{\End}{\mathrm{End}}            % Endomorphims ring
\DeclareMathOperator{\Hom}{\mathrm{Hom}}            % Homomorphism
\DeclareMathOperator{\Mor}{\mathrm{Mor}}            % Morphism
\DeclareMathOperator{\codim}{\mathrm{codim}}        % codimension

\DeclareMathOperator{\CH}{\mathrm{CH}}              % Chow group
\DeclareMathOperator{\ch}{\mathrm{c}}               % Chern class

\DeclareMathOperator{\QK}{\mathrm{K}}               % Quillen K-group functor
\newcommand{\MK}[1]{\mathrm{K}^{M}_{#1}}            % Milnor K-group functor
\DeclareMathOperator{\CM}{\mathrm{M}}               % Rost cycle module
\DeclareMathOperator{\CK}{\mathrm{C}}               % Cycle complex

\newcommand{\one}{\mathbb{I}}

\DeclareMathOperator{\HM}{\mathrm{H}}               % (Co-)homology

\DeclareMathOperator{\Image}{\mathrm{Im}}           % image
\DeclareMathOperator{\res}{\mathrm{res}}            % restriction map
\DeclareMathOperator{\id}{\mathrm{id}}              % identity map

\newcommand{\ep}{\epsilon}

\newcommand{\hh}{\mathtt{h}}   % oriented cohomology

\DeclareMathOperator{\Pcat}{{\mathcal P}}                     % category of vector bundles of finite rank
\DeclareMathOperator{\GSm}{\mathit{G}\text{-}\mathrm{Sm}}  % category of smooth varieties
\DeclareMathOperator{\Ess}{\mathit{G}\text{-}\mathrm{Ess}}  % Category of essentially smooth
\DeclareMathOperator{\Ab}{\mathfrak{Ab}}                      % category of abelian groups
\newcommand{\Db}{\mathrm{D}^{b}}                     % bounded derived category
\newcommand{\Kb}{\mathrm{K}^{b}}                     % category of bounded complexes

\newcommand{\ul}[1]{\underline{#1}}

\newcommand{\ie}{\textsl{i.e.}\ }
\newcommand{\cf}{\textsl{cf.}\ }
\newcommand{\eg}{\textsl{e.g.}\ }
\newcommand{\locit}{\textsl{loc.cit.}\ }

\begin{document}

\title{Equivariant pretheories and invariants of torsors}

\author{Stefan Gille}
\author{Kirill Zainoulline}

\address{Stefan Gille, Mathematisches Institut, Universit\"at M\"unchen, 
Theresienstrasse 39,
         80333 M\"unchen, Germany}
\email{gille@mathematik.uni-muenchen.de}

\address{Kirill Zainoulline, Department of Mathematics and Statistics, 
University of Ottawa,
         585 King Edward, Ottawa ON K1N 6N5, Canada}
\email{kirill@uottawa.ca}


%%%%%%%%%%%%%%%%%%%%%%%%%%%%%%%%%%%%%%%%%%%%

\begin{abstract}
Using the notion of an equivariant pretheory 
we generalize a theorem of Karpenko-Merkurjev
on $G$-torsors and rational cycles; to every
$G$-torsor and equivariant pretheory we associate
a commutative ring which in the case of Chow groups encodes
the information concerning the $J$-invariant
and in the case of Grothendieck's $K_{0}$--indexes of the Tits algebras.
\end{abstract}

%%% Document STARTS here

\maketitle

\section{Introduction}

In the present paper we introduce and study the notion of a (graded) equivariant
pretheory.
Roughly speaking, it is defined to be a contravariant functor from
the category of $G$-varieties, where $G$ is an algebraic group, to (graded) abelian groups which
satisfies localization and homotopy invariance properties.
All known examples of equivariant oriented cohomology theories (equivariant Chow groups, $K$-theory, algebraic cobordism, etc.) are pretheories in our sense.

\smallbreak

We generalize the equivariant Chow groups of Edidin-Graham by introducing equivariant
(co)homology theory with coefficients in a Rost cycle module. We also prove a version of Merkurjev's equivariant $K$-theory
spectral sequence for equivariant cycle homology. This provides many new
examples of equivariant pretheories. 

\smallbreak

One of the key results of Karpenko-Merkurjev \cite[Thm.~6.4]{KaMe06} tells
us that the characteristic subring  of the Chow ring of a variety of Borel
subgroups of a split linear algebraic group $G$ 
is contained in the image of the restriction map, i.e. always consists
of rational cycles. 
This fact plays a fundamental role in
computations of canonical/essential dimensions, discrete motivic invariants of~$G$
and in the study of splitting properties of $G$-torsors.

\smallbreak

In the present paper we generalize this result to an arbitrary equivariant pretheory (see Theorem~\ref{mainThm}).
In particular, we obtain versions of \cite[Thm.~6.4]{KaMe06}
for Grothendieck's $K_0$ and algebraic cobordism $\Omega$ of Levine-Morel.

\smallbreak

As an application we define for any equivariant pretheory~$\hh$ and $G$-torsor~$E$
a commutative ring~$\hat \hh_B(E)$
(see Def.~\ref{geninv}). If $E$ is generic and
$\hh$ is either the Chow ring $\CH^*$ or Grothendieck's $K_{0}$ or algebraic cobordism $\Omega$,
this ring coincides with the cohomology ring $\hh(G)$ of~$G$.
In general, it is always a quotient of $\hh(G)$ 
which in the case of the Chow ring
is related to the motivic $J$-invariant of $E$ and in the case of
$K_{0}$ -- to the indexes of the Tits algebras of $E$. This provides 
a fascinating link between these two discrete invariants, totally unrelated at the first sight:
one observes that the $p$-exceptional
degrees of V.~Kac for Chow groups \cite{Kc85} play the same role as the maximal
Tits indexes for $K_{0}$ \cite{MePaWa96}.

\smallbreak

The paper is organized as follows: In the first two sections we introduce the notion of
an equivariant pretheory and provide several examples including equivariant cycle
(co)homology.
In Section~4 we generalize the result of Karpenko-Merkurjev to an arbitrary equivariant pretheory.  
In the last section we provide applications to equivariant oriented cohomology theories
(Chow groups, Grothendieck's $K_0$ and algebraic cobordism of Levine-Morel). Appendix is devoted
to the construction of a spectral sequence for cycle homology 
which generalizes the long exact localization sequence.

\medbreak

\begin{ntt}[\bf Notations]\label{notation}
Unless otherwise indicated, all schemes/varieties are defined over the base field $k$.
By a scheme over a field $k$ ($k$-scheme) we mean a reduced separated Noetherian scheme over $k$.
By a variety over a field $k$ ($k$-variety) we mean a quasi-projective scheme over $k$
(note that it has to be of finite type over $k$).
If $l/k$ is a field extension and~$X$ is a $k$-scheme, we define
$X_{l}=X\times_{\Spec k} {\Spec l}$ to be the respective base change. By $\pt$ we denote $\Spec k$.

\smallbreak

By an algebraic group we mean an affine smooth group scheme over $k$.
By a subgroup we always 
understand a closed algebraic subgroup. 
By an action of an algebraic group $G$ on a scheme $X$ 
we mean a morphism $G\times_{\Spec k} X\to X$ of schemes over $k$
(all group actions are assumed to be on the left), subject
to the usual axioms, see~\cite[Def.\ 0.3]{GIT}.
By a $G$-scheme we mean a scheme $X$ endowed with
an action of an algebraic group $G$. 

\smallbreak
We denote by $\GSm_k$ the category of smooth $G$-varieties over $k$ with equivariant $G$-morphisms.
A localization of a smooth variety over $k$ is called essentially smooth.
We denote by $\Ess_k$ the category of 
essentially smooth $G$-schemes over $k$ with 
$G$-equivariant flat morphisms.
We denote by $\Ab$ the category of abelian groups.
\end{ntt}

%%%%%%%%%%%%%%%%%%%%%%%%%%%%%%%
%%%%%%%%%%%%%%%%%%%%%%%%%%%%%%%
%%%%%%%%%%%%%%%%%%%%%%%%%%%%%%%

\section{Equivariant pretheories.}
\label{Torsor-PreThSect}

In the present section we introduce the notion of 
a (graded) equivariant pretheory and
provide several examples.

\smallbreak
Let $G$ be an algebraic group over a field $k$.
Consider a contravariant functor 
from the category of smooth $G$-varieties over $k$ to
the category of abelian groups
$$
\hh_G\colon \GSm_k \too \Ab,\quad X\mapsto \hh_G(X).
$$ 
Given $X$, $Y\in \Ess_k$ and 
a $G$-equivariant map $f\colon X\to Y$
the induced functorial map $\hh_{G}(Y)\to \hh_{G}(X)$ is
called a {\it pull-back} and is denoted by  $f^{*}_{G}$.

\begin{dfn}\label{dfnpre}
The functor 
$\hh_G\colon \GSm_k \to \Ab$ 
is called a {\it $G$-equivariant pretheory} over $k$ 
if it satisfies
the following two axioms:

\begin{itemize}
\item[H.]
(homotopy invariance) 
For a $G$-equivariant 
map $p\colon \A_k^{n} \to \pt$ 
(where $G$ acts trivially on $\pt$)
the induced pull-back 
$$
p^\ast_G\colon \hh_G(\pt)\too\hh_G(\A^n_k)
$$
is an isomorphism.

\smallbreak

\item[L.]
(localization) 
For a smooth $G$-variety $X$ and 
a $G$-equivariant open embedding
$\iota\colon U\hookrightarrow X$ 
the induced pull-back
$$
\iota^{\ast}_G\colon \hh_G(X)\too\hh_G(U)
$$
is surjective.
\end{itemize}
\end{dfn}

Let $\LU$ be a $G$-scheme over $k$ such that
$\LU$ is the localization of a smooth irreducible $G$-variety $X$ with respect to
$G$-equivariant open embeddings 
$f_{ij}\colon U_j\to U_i$, $U_i\subset X$, i.e.
$\LU=\varprojlim_{f_{ij}} U_i$. Observe that $\LU$ is essentially smooth over $k$.

\smallbreak

Let $\bar\hh_G(\LU)$ denote the induced colimit  $\varinjlim_{(f_{ij})^*_G} \hh_G(U_i)$.
Note that the canonical maps 
$\hh_G(U_i) \to \bar\hh_G(\LU)$ 
are surjective 
by the localization property (L).

\begin{dfn} We call $\hh_G$
an {\em essential} $G$-equivariant pretheory if $\hh_G$ 
can be extended to the category $\Ess_k$ of essentially smooth
$G$-schemes over $k$ with $G$-equivariant flat morphisms, i.e.  
$$
\hh_G\colon \Ess_k\too \Ab,
$$
such that the following additional axiom holds:

\begin{itemize}
\item[C.] Given $\LU$ as above, the map induced by flat pull-backs  $\hh_G(U_i)\to \hh_G(\LU)$ 
$$\bar\hh_G(\LU)\to \hh_G(\LU)$$ 
is surjectve.
\end{itemize}

Note that (C) holds if and only if 
the induced pull-back 
$\hh_G(U_i) \to \hh_G(\LU)$ 
is surjective for some $i$.
\end{dfn}

\begin{ex}[Equivariant $K$-theory]
We recall definitions and basic properties of equivariant $K$-groups
as defined by Thomason~\cite{Th87}, 
see also the survey article~\cite{Me05} of Merkurjev.

\smallbreak

Let~$G$ be an algebraic group over $k$ 
and~let $X$ be a smooth $G$-variety.
Then the category $\Pcat (G,X)$ of locally free
$G$-modules on $X$ 
(in the sense of Mumford~\cite[I, \S 3]{GIT})
is an exact category. 
Following Thomason~\cite{Th87} 
one defines the $i$-th $G$-equivariant $K$-group 
$\QK_{i}(G,X)$ as Quillen's $i$-th $K$-group 
of the exact category $\Pcat (G,X)$. 

\smallbreak

Let $\hh_{G}(X)=\QK_{0}(G,X)$. 
Then according to \cite[Thm.\ 2,7 and Lem.\ 4.1]{Me05} 
it satisfies localization and homotopy invariance, 
and by \cite[52.F]{AGQ} it satisfies (C). 
Hence, it provides an example of an essential $G$-equivariant pretheory. 
\end{ex}

\begin{ex}[Equivariant cobordism]
This theory has been recently defined by 
Heller and Malag\'on-L\'opez~\cite{HeMa-Lo10}.

\smallbreak

Assume that $char(k)=0$. 
Consider the ring  $\Cob_*(X)$ of algebraic cobordism
of a smooth $k$-variety~$X$ 
as defined by Levine and Morel~\cite{AC}.
Since $\Cob_{i}(X)$ does not vanish for~$i$ big enough 
(as Chow groups do) 
one can not copy word by word the definition 
of equivariant Chow groups 
given by Edidin and Graham~\cite{EdGr98},
see also Section~\ref{EqCycleHomSect}.

\smallbreak

Instead Heller and Malag\'on-L\'opez consider \cite{HeMa-Lo10} 
(what they call) {\it good systems of representations}. 
These are families of pairs
$(V_{i},U_{i})_{i\in\N}$ of vector spaces 
with $U_ {i}\subseteq V_{i}$ endowed with 
an action of an algebraic group~$G$
such that 
\begin{itemize}
\item[(i)] 
$G$ acts freely on~$U_{i}$ and $U_{i}\to U_{i}/G$ is a $G$-torsor, 
\item[(ii)]
$V_{i+1}=V_{i}\oplus W_{i}$ for some $k$-subspace $W_{i}$, 
such that $U_{i}\oplus W_{i}\subseteq U_{i+1}$,
\item[(iii)] 
$\sup\dim V_{i}=\infty$, and 
\item[(iv)] 
$\codim_{V_{i}}(V_{i}\setminus U_{i})<\codim_{V_{i+1}}(V_{i+1}\setminus U_{i+1})$, 
where we consider $V_{i}$ as an affine space over $k$. 
\end{itemize}
Observe that assumption~(i) ensures that
the quotient $X\times^G U:=(X\times_k U)/G$ 
is a quasi-projective variety over $k$
(see \cite[Prop.\ 23]{EdGr98}). 
Moreover, it is smooth over $k$ by the descent, 
since $X\times_k U \to X\times^G U$ is faithfully flat.

\smallbreak

Let $G$ be connected.  Then the $n$-th
{\it equivariant cobordism group} of a 
smooth $G$-variety~$X$ is defined by
$$
\Cob_{n}^{G}(X)\, := 
\varprojlim_i \Cob_{n-\dim G+\dim U_{i}} (X\times^{G}U_{i})\, .
$$
This is well defined, see~\cite[Cor.\ 3.4]{HeMa-Lo10}, 
and the functor
$$
\hh_G\colon X\,\longmapsto\,
\bigoplus_{n\in\zz}\Cob_{n}^{G}(X)
$$
satisfies the localization and homotopy invariance axioms by
[\locit Thm.\ 4.2 and Cor.\ 4.6].
Hence, it provides an example of a $G$-equivariant pretheory.
\end{ex}

A further example is 
the equivariant Chow-theory of Edidin and Graham~\cite{EdGr98}. 
We consider this later (see Example~\ref{equivChow}) 
when we take a closer look at equivariant cycle (co)homology.

\smallbreak

There is also a graded version of a $G$-equivariant pretheory
\begin{dfn}
\label{defgrpre}
A pair of varieties $(X,U)$ is called a $G$-pair 
if $X\in\GSm_k$ and $U\subseteq X$
is a $G$-equivariant open subvariety.
Consider the category of $G$-pairs over $k$ 
with $G$-equivarant morphisms of pairs.

A contravariant functor
$$
(X,U)\,\longmapsto\,\hh_{G}^{\ast}(X,U)
$$
from the category of $G$-pairs to graded abelian groups 
is called a {\it graded $G$-equivariant pretheory} 
if it satisfies (H)~homotopy invariance, and for any $G$-pair $(X,U)$ 
there is a long exact localization sequence
$$
\xymatrix{
\ldots \ar[r] & \hh_G^i(X) \ar[r]^-{\iota^{\ast}_{G}} &
    \hh_G^i(U) \ar[r]^-{\partial} &  \hh_G^{i+1}(X,U)
           \ar[r] & \ldots\, ,
}
$$
where $\iota\colon U\hookrightarrow X$ is the corresponding
$G$-equivariant open embedding, 
and we have set
$\hh_{G}^{\ast}(Y):=\hh_{G}^{\ast}(Y,Y)$. It is called a {\em graded essential $G$-equivariant pretheory} 
if given an inverse limit $(\mathcal{X},\LU)=\varprojlim_i (X_i, U_i)$
of $G$-equivariant open embeddings of pairs, there is the induced surjection
$$
\varinjlim_i \hh_G^*(X_i,U_i) \too \hh_G^*(\mathcal{X},\LU).
$$
\end{dfn}

%%%%%%%%%%%%%%%%%%%%%%%%%%%%%%%%%%%%%%%%%%%%%%%%%%%%
%%%%%%%%%%%%%%%%%%%%%%%%%%%%%%%%%%%%%%%%%%%%%%%%%%%%%
%%%%%%%%%%%%%%%%%%%%%%%%%%%%%%%%%%%%%%%%%%%%%%%%%%%%

\section{Equivariant cycle (co)homology}
\label{EqCycleHomSect}

\noindent
In this section we generalize the equivariant Chow groups 
of Edidin and Graham~\cite{EdGr98}. 
This theory has been considered for the cycle module
Galois cohomology by Guillot~\cite{Gu07}.
We will use freely Rost's~\cite{Ro96} theory of cycle modules for
which we refer also to the book~\cite{AGQ}
of Elman, Karpenko and Merkurjev, as well as to the article
of D\'eglise~\cite{Deg06} where several important properties of the
generalized ``intersection'' product in cycle cohomology are
proven (defined in~\cite[Sect.\ 14]{Ro96}).

Since Rost's theory for algebraic spaces is not yet developed we
have to restrict ourseives to quasi-projective schemes, i.e. to varieties.
This assumption guarantees that certain quotients by groups
actions which we consider here do exist.

\begin{ntt}[{\bf Equivariant cycle homology}]
To fix notations we recall briefly the definition of cycle homology.
A {\it cycle module} over the field~$k$ is a (covariant) functor
$\CM_{\ast}$
from the category of field extensions of~$k$ to the category of graded
abelian groups subject to several axioms, see~\cite[Sects.\ 1,2]{Ro96}.
The prototype of such a functor is {\it Milnor $K$-theory} $\MK{\ast}$,
and by the very definition
$\CM_{\ast}(E)=\bigoplus\limits_{i\in\zz}\CM_{i}(E)$ is a graded
$\MK{\ast}(E)$-module for all field extensions $E\supseteq k$.

\smallbreak

Given a $k$-variety~$X$ (not necessarily smooth) and a cycle module~$\CM_{\ast}$ over~$k$
Rost~\cite{Ro96} has defined a complex, the so called {\it cycle complex}
(generalizing a construction of Kato~\cite{Ka86} for Milnor $K$-theory):
$$
\xymatrix{
\ldots\; \ar[r] & {\bigoplus\limits_{x\in X_{(2)}}}\CM_{n+2}(k(x)) \ar[r]^-{d_{2}} &
   {\bigoplus\limits_{x\in X_{(1)}}}\CM_{n+1}(k(x)) \ar[r]^-{d_{1}} &
      {\bigoplus\limits_{x\in X_{(0)}}}\CM_{n}(k(x)),
}
$$
where $X_{(i)}\subseteq X$ denotes the set of points of dimension~$i$ in~$X$.
We denote this complex $\CK_{\smb}(X,\CM_{n})$ and consider it as a homological
complex with the direct sum $\bigoplus\limits_{x\in X_{(i)}}\CM_{n+i}(k(x))$ in
degree~$i$.

\smallbreak

The $i$-th cycle homology group~$\HM_{i}(X,\CM_{n})$ of~$\CM_{n}$
over~$X$ is then defined as $\HM_{i}(\CK_{\smb}(X,\CM_{n}))$. Note that
there is a natural isomorphism $\HM_{i}(X,\MK{-i})\simeq\CH_{i}(X)$ for
all $i\geq 0$, where we have set $\MK{-i}\equiv 0$ for $i<0$.

\smallbreak

\label{EchSubSect}
To introduce the equivariant cycle homology we adapt
the definition of equivariant Chow groups due to Edidin
and Graham~\cite{EdGr98}, see also Guillot~\cite{Gu07} and
Totaro~\cite{To97}.

\smallbreak

Let $G$ be an algebraic group over $k$ of dimension~$s$ and $X$~a
$G$-variety. To define the $i$-th cycle homology group with coefficients
in the cycle module~$\CM_{\ast}$ we chose a linear representation~$V$
of~$G$, such that there is an open subscheme $U\hookrightarrow V$ with
$\codim_{V}(V\setminus U)\geq c=\dim X$ on which~$G$ acts freely. By shrinking~$U$
we can moreover assume that $U\too U/G$ is a principal bundle. The later
assumption assures that $X\times^{G}U:=(X\times_{k}U)/G$ exists
in the category of $k$-varieties, see~\cite[Prop.\ 23]{EdGr98}
(recall that we assume that~$X$ is quasi-projective, see \ref{notation}). We
call the pair $(U,V)$ an {\it $(X,G)$-admissible pair} for the
$G$-variety~$X$. Note that for a finite number of $G$-varieties there always
exist a pair~$(U,V)$ which is admissible for all of them.
\end{ntt}

\begin{dfn}
\label{EqCyclehomDef}
Let~$G$ be an algebraic group over $k$.
Then the $i$-th {\it $G$-equivariant cycle homology group}
with values in the cycle module~$\CM_{\ast}$ over $k$ is defined as
$$
\HM^{G}_{i}(X,\CM_{\ast})\, :=\;
      \HM_{i +l-s}(X\times^{G}U,\CM_{\ast -(l-s)})\, ,
$$
where $s=\dim G$ and $(U,V)$ is a $(X,G)$-admissible pair
with $\dim V =l$ and $\codim_{V}V\setminus U\geq\dim X$.
\end{dfn}

It remains to check that this definition does not depend on the choice of the
$(X,G)$-admissible pair $(U, V)$.
Since $\HM_{j}(Y,\CM_{\ast})=0$ for any
$k$-variety~$Y$ if $j>\dim Y$ this can be proven as
for equivariant Chow groups using (the so called) Bogomolov's
double filtration argument, see~\cite{EdGr98} or~\cite{To97}.
We recall briefly the details:

\smallbreak

Let $U_{1}\subset V_{1}$ be another
$(X,G)$-admissible pair with $l_{1}=\dim V_{1}$. Then there
exists an open subvariety~$W$ of $V_{1}\oplus V$, which contains
$U_{1}\oplus V$ and $V_{1}\oplus U$ as open subvarieties, and such
that~$G$ acts on~$W$ with principal bundle quotient~$W/G$. Then
the quotient $X\times^{G}W$ exists in the category of varieties.

\smallbreak

We use the following fact from~\cite[Prop.\ 2 and Lem.\ 1]{EdGr98},
which is not hard to verify if $U\too U/G$ is a trivial $G$-torsor
and follows in general by descent from the ``trivial'' case.

\begin{lem}
\label{P-Lem}
Let $G,U,V$ and~$X$ be as above and $f\colon X\to Y$ be a $G$-morphism.
Denote $f\times^{G}\id_{U}\colon X\times^{G}U\to Y\times^{G}U$ the
induced morphism.

Consider the following properties {\bf P} of~$f$: vector bundle,
flat, smooth, proper, regular immersion,
open or closed immersion. Then if~$f$ as property~{\bf P} implies that
also $f\times^{G}\id_{U}$ has property~{\bf P}. 
\end{lem}

\smallbreak

\begin{ntt}
We also need the following fact. Let
$$
\xymatrix{
X' \ar[r]^-{f'} \ar[d]_-{g'} & Y' \ar[d]^-{g} \\
X \ar[r]^-{f} & Y
}
$$
be a cartesian square of $G$-varieties and $(U,V)$ a pair which
is admissible for all varieties in the square. Then there is a natural
and unique morphism
$$
X'\times^{G}U\too X\times^{G}U\,\times_{Y\times^{G}U}\,
    Y'\times^{G}U\, ,
$$
which is an isomorphism if $U\too U/G$ is a trivial $G$-torsor.
Hence by descent, see \eg~\cite[Thm.\ 2.55 and Lem.\ 4.44]{FAG}, it is an isomorphism in general.
\end{ntt}

\smallbreak

The morphism
$$
\id_{X}\times^{G}p_{U}\colon X\times^{G}(V_{1}\oplus U)\too X\times^{G}U\, ,
$$
where $p_{U}\colon V_{1}\oplus U\too U$ is the projection,
and the inclusion
$$
\id_{X}\times^{G}\iota_{U}\colon X\times^{G}(V_{1}\oplus U)\too X\times^{G}W\, ,
$$
where $\iota_{U}\colon W\hookrightarrow V_{1}\oplus U$, induce homomorphisms
$$
\xymatrix{
{\HM_{i+l_{1}+l-s}(X\times^{G}W,\CM_{n-l_{1}-l+s})} \ar[rd]^-{(\id_{X}\times^{G}\iota_{U})^{\ast}} &
\\
 & {\HM_{i+l_{1}+l-s}(X\times^{G}(V_{1}\oplus U),\CM_{n-l_{1}-l+s})} \rlap{\, .}
\\
{\HM_{i+l-s}(X\times^{G}U,\CM_{n-l+s})} \ar[ru]_-{(\id_{X}\times^{G}p_{U})^{\ast}}
}
$$
Both maps are isomorphisms. The first since the dimension of the closed
complement of $X\times^{G}(V_{1}\oplus U)$ in $X\times^{G}W$ is smaller
than $i+l_{1}+l-s$, and the second by homotopy invariance (recall
that $\id_{X}\times^{G}p_{U}$ is a vector bundle by Lemma~\ref{P-Lem}).
Similarly we have an isomorphism
$$
\HM_{i+l_{1}+l-s}(X\times^{G}W,\CM_{n-l_{1}-l+s})\,\xrightarrow{\;\simeq\;}\,
       \HM_{i+l_{1}-s}(X\times^{G}U_{1},\CM_{n-l_{1}+s})\, ,
$$
and hence both pairs define natural isomorphic groups.

\begin{ntt}[Pull-backs and push-forwards]
Let now $f\colon X\too Y$ be a morphism of 
$G$-varieties, and $(U,V)$ and $(U_{1},V_{1})$ two
$(X,G)$- and $(Y,G)$-admissible pairs as
above. We set $d=\dim X -\dim Y$. Then we have by the
functorial properties of push-forward and pull-back maps
in cycle homology (using Lemma~\ref{P-Lem} to see that
the maps in question are defined) two commutative diagrams:

\smallbreak

(i)\ \ 
If $f$ is flat of constant relative dimension or~$Y$ is a
smooth $k$-variety then
$$
\xymatrix{
\HM_{\ast +l-s}(Y\times^{G}U,\CM_{n-l+s}) \ar[rr]^-{(f\times^{G}\id_{U})^{\ast}}
    \ar@{-}[d]^-{\simeq} & & \HM_{\ast +d+l-s}(X\times^{G}U,\CM_{n-l-d+s}) \ar@{-}[d]_-{\simeq}
\\
\HM_{\ast +l_{1}-s}(Y\times^{G}U_{1},\CM_{n-l+s}) \ar[rr]^-{(f\times^{G}\id_{U_{1}})^{\ast}} & &
            \HM_{\ast +d+l_{1-s}}(X\times^{G}U_{1},\CM_{n-l_{1}-d+s})
}
$$
is a commutative diagram whose column arrows are natural isomorphisms, and

\smallbreak

(ii)\ \ 
If~$f$ is proper there is another commutative diagram
$$
\xymatrix{
\HM_{\ast +l-s}(X\times^{G}U,\CM_{n-l+s}) \ar[rr]^-{(f\times^{G}\id_{U})_{\ast}}
    \ar@{-}[d]^-{\simeq} & & \HM_{\ast +l-s}(X\times^{G}U,\CM_{n-l+s}) \ar@{-}[d]_-{\simeq}
\\
\HM_{\ast +l_{1}-s}(X\times^{G}U_{1},\CM_{n-l_{1}+s}) \ar[rr]^-{(f\times^{G}\id_{U_{1}})_{\ast}} & &
            \HM_{\ast +l_{1}-s}(X\times^{G}U_{1},\CM_{n-l_{1}+s}) \rlap{\, ,}
}
$$
whose column arrows are again natural isomorphisms.

\smallbreak

Let now $f\colon X\to Y$ be a morphism of $G$-varieties and
$(U,V)$ a $(G,X)$- and $(G,Y)$-admissible pair.

\smallbreak

(i)\ \ 
If either~$f$ is flat of constant relative dimension
or~$Y$ is smooth, we define the pull-back
morphism 
$$
f_{G}^{\ast}\colon \HM^{G}_{i}(Y,\CM_{n})\too\HM^{G}_{i+d}(X,\CM_{n+d})$$ as
$$
(f\times^{G}\id_{U})^{\ast}\colon \HM_{i+l-s}(Y\times^{G}U,\CM_{n-l+s})\too
           \HM_{i+l-s+d}(X,\CM_{n+l-s-d})\, ,
$$
where $d=\dim X-\dim Y$.

\smallbreak

(ii)\ \ 
If~$f$ is proper, we define the push-forward morphism
$$f_{G\ast}\colon \HM^{G}_{i}(X\CM_{n})\too\HM^{G}_{i}(Y,\CM_{n})$$ as
$$
(f\times^{G}\id_{U})_{\ast}\colon \HM_{i}(X,\CM_{n})\too
           \HM_{i}(Y,\CM_{n})\, .
$$

\smallbreak

With these definitions $G$-equivariant
cycle homology is a contravariant functor on the category
of smooth and $G$-varieties and also
covariant for proper morphisms. Moreover, if
$$
\xymatrix{
X' \ar[r]^-{g'} \ar[d]_-{f'} & X \ar[d]^-{f}
\\
Y' \ar[r]^-{g} & Y 
}
$$
is a cartesian square with~$f$ proper and~$g$ flat of constant
relative dimension or all varieties in the diagram are smooth
over~$k$ then we have
$$
f'_{G\ast}\circ g^{\ast}_{G}\, =\, g'^{\ast}_{G}\circ f_{G\ast}\, .
$$
\end{ntt}

\begin{ntt}[Axioms]\label{axioms}
There is the localization sequence
$$
\xymatrix{
\ldots\, \ar[r]^-{\iota_{G\ast}} & \HM_{i}^{G}(X,\CM_{n}) \ar[r]^-{j^{\ast}_{G}} &
  \HM_{i}^{G}(X\setminus Z,\CM_{n}) \ar[r]^-{\partial} & \HM_{i-1}^{G}(Z,\CM_{n})
         \ar[r]^-{\iota_{G\ast}} & \,\ldots
}
$$
for a closed $G$-embedding $\iota\colon  Z\hookrightarrow X$
with open $G$-equivariant complement $j\colon X\setminus Z\hookrightarrow X$
which follows from the localization sequence in ordinary cycle homology. And
if $\pi\colon E\too X$ is a $G$-vector bundle of rank~$r$, \ie a $G$-linear
bundle, then $$\pi\times^{G}\id_{U}\colon E\times^{G}U\too X\times^{G}U$$
is a vector bundle, see~\cite[Lem.\ 1]{EdGr98}, and therefore the
pull-back $$\pi^{\ast}\colon \HM_{i}^{G}(X,\CM_{n})\too\HM_{i+r}^{G}(E,\CM_{n-r})$$
is an isomorphism by homotopy invariance of cycle homology.
Finally, $\HM_{i}^{G}(X,\CM_{n})$ can be extended to $\Ess_k$ and satisfies (C) by \cite[52.F]{AGQ}
or \cite[p.320]{Ro96}. 
\end{ntt}

\begin{ex}
\label{G-cohTorsorEx}
If $f\colon X\to Y=X/G$ is a $G$-torsor then we have a natural
isomorphism
$$
\HM_{i-s}(Y,\CM_{n+s})\,\simeq\,\HM_{i}^{G}(X,\CM_{n})
$$
where $s=\dim X-\dim Y=\dim G$, for all $i\in\N$ and $n\in\zz$.
This can be seen as follows, \cf~\cite[Expl.\ 2.3.2]{Gu07}:

\smallbreak

We choose a $(X,G)$-admissible pair $(U,V)$. Let $l=\dim V$.
Then the closed complement of $X\times^{G}U=(X\times_{k}U)/G$
in $(X\times_{k}V)/G$ has dimension less than $i+l-s$ and
therefore we have by the localization sequence an isomorphism
$$
\HM_{i+l-s}((X\times_{k}V)/G,\CM_{n-(l-s)})
\xrightarrow{\; j^{\ast}\;}\,
\HM_{i+l-s}(X\times^{G}U,\CM_{n-(l-s)})\, =\,
\HM_{i}^{G}(X,\CM_{n})\, ,
$$
where $j\colon (X\times_{k}U)/G\hookrightarrow (X\times_{k}V)/G$
is the  corresponding open immersion (note that the target
of~$j$ exists in the category of varieties since by assumption
$X\to X/G$ is a $G$-torsor). By~\cite[Lem.\ 1]{EdGr98} we know that
$(X\times_{k}V)/G\to X/G$ is a vector bundle of rank~$l$ and
so by homotopy invariance we have
$$
\HM_{i+l-s}((X\times_{k}V)/G,\CM_{n-(l-s)})\,\simeq\,
   \HM_{i-s}(X/G,\CM_{n+s})\, .
$$
\end{ex}

\begin{ntt}[Restriction map]
If $G_{1}\subseteq G$ is a closed subgroup and~$X$ a 
$G$-variety over $k$, we can choose a $(X,G)$- and
$(X,G_{1})$-admissible pair $(U,V)$. Then we have a morphism
of $k$-varieties $(X\times_{k}U)/G_{1}\to (X\times_{k}U)/G$
which induces (via pull-back) a homomorphism of equivariant
cycle homology groups
$$
\res_{G_{1}}^{G}\colon \HM^{G}_{i}(X,\CM_{n})\too
     \HM^{G_{1}}_{i}(X,\CM_{n})
$$
for any cycle module~$\CM_{\ast}$ called {\it restriction
homomorphism}. In particular if $G_{1}$ is the trivial
group we have a (forgetful) morphism $\HM^{G}_{i}(X,\CM_{n})\too
\HM_{i}(X,\CM_{n})$ from $G$-equivariant cycle homology to
ordinary cycle homology.
\end{ntt}

\begin{ntt}[The first Chern class]
Let  $\pi\colon L\to X$ be a $G$-equivariant line bundle with
zero section $\sigma\colon X\to L$. Then~$\sigma$ is a
$G$-equivariant closed embedding and so induces a
morphism $\sigma\times^{G}\id_{U}\colon X\times^{G}U\to L\times^{G}U$
which is the zero section of the line bundle (see Lemma~\ref{P-Lem})
$$
\pi\times^{G}\id_{U}\colon L\times^{G}U\too X\times^{G}U\, .
$$
The {\it first Chern class} (or also called {\it Euler class})
of~$L$ is then defined as the operator
$$
\ch_{1}(L)^{G}\, :=\, (\pi^{\ast}_{G})^{-1}\circ\,\sigma_{G\ast}
\colon\HM_{\ast}^{G}(X,\CM_{\ast})\too\HM_{\ast -1}^{G}(X,\CM_{\ast +1})\, .
$$

\smallbreak

This map commutes with push-forwards and pull-backs. More
precisely, assume we have a cartesian square of $G$-equivariant
morphisms, where~$L$ and~$L'$ are $G$-equivariant line bundles
over~$X$ and~$X'$, respectively:
$$
\xymatrix{
L' \ar[r]^{f'} \ar[d]_-{\pi'} & L \ar[d]^-{\pi}
\\
X' \ar[r]^-{f} & X \rlap{\, .}
}
$$
Denote by~$\sigma$ and~$\sigma'$ the zero sections
of~$L$ and~$L'$, respectively. These are also $G$-equivariant
morphisms and so we get a cartesian square
$$
\xymatrix{
L'_{G} \ar[rr]^-{f'_{G}} \ar@/^/[d]^-{\pi'_{G}} & & L_{G}
       \ar@/^/[d]^-{\pi_{G}}
\\
X'_{G} \ar[rr]^-{f_{G}} \ar@/^/[u]^-{\sigma'_{G}} & & X_{G}
       \ar@/^/[u]^-{\sigma_{G}} \rlap{\, ,}
}
$$
where we have set $Y_{G}:=Y\times^{G}U$ for any $G$-variety~$Y$,
$g_{G}:=g\times^{G}\id_{U}$ for all $G$-equivariant morphisms
$g\colon Y'\too Y$, and assumed that the pair~$(U,V)$ is admissible
for all varieties in question. A straightforward computation using
this diagram, see~\cite[Prop.\ 53.3]{AGQ} for the analogous
result in ordinary cycle homology, shows
$$
\ch_{1}^{G}(L')\circ\, f^{\ast}_{G}\, =\, f^{\ast}_{G}\circ\,\ch_{1}^{G}(L)\quad \text{ if }f \text{ is flat, and }
$$
$$
\ch_{1}^{G}(L)\circ\, f_{G\ast}\, =\, f_{G\ast}\circ\,\ch_{1}^{G}(L')\quad\text{ if } f \text{ is proper.}
$$
\end{ntt}

\begin{rem}
(i)\ \ 
The Chern class homomorphism
$$
\ch_{1}^{G}(L)\colon \HM_{i}^{G}(X,\MK{-i})=\CH_{i}^{G}(X)\too
    \CH_{i-1}^{G}(X)=\HM_{i-1}^{G}(X,\MK{-i+1})
$$
coincides with the first Chern class defined in
Edidin and Graham~\cite[Sect.\ 2.4]{EdGr98}.

\smallbreak

(ii)\ \ 
As for ordinary cycle homology in~\cite{AGQ} one can use the
Euler class of a vector bundle to prove the projective bundle
theorem and then use this to define the higher Chern classes in
$G$-equivariant cycle homology. We leave this to
the interested reader.
\end{rem}

\begin{ntt}[An equivariant spectral sequence]
\label{eqspecsec}
We provide now a version of
Merkurjev's~\cite{Me05} equivariant $K$-theory spectral
sequence for equivariant cycle homology. 

\smallbreak

Let $T$ be a $k$-split torus of rank~$m$, and $\chi\colon T\to\Gm$
a character. The algebraic group~$T$ acts via~$\chi$ on the
affine line $\A^{1}_{k}$ defining a $T$-equivariant line
bundle $\A^{1}_{k}\to\pt$ which we denote by~$L(\chi)$
(trivial action of~$T$ on the base point~$\pt$).
If~$p\colon X\to\pt$ is a $T$-scheme we denote the
pull-back~$p^{\ast}L(\chi)$ by~$L_{X}(\chi)$. This is also a
$T$-equivariant vector bundle.

\smallbreak

Let $X$
be a  $G$-variety, where~$G$ is an algebraic group over $k$, and $T\subseteq G$ a split
torus of rank~$m$. Let $\chi_{1},\ldots ,\chi_{m}$ be a
basis of the character group~$T^{*}=\Hom (T,\Gm)$, and let~$T$ act
on the affine space~$\A^{m}_{k}=\Spec k[x_{1},\ldots ,x_{m}]$
by
\begin{equation}\label{actT}
t\cdot (a_{1},\ldots ,a_{m})\,\longmapsto\,
  (\chi_{1}(t)\cdot a_{1},\,\ldots\, ,\chi_{m}(t)\cdot a_{m})\, ,
\end{equation}
and on $\A^{m}_{X}=X\times_{k}\A^{m}_{k}$ diagonally. Let
$Z_{i}\subset\A^{m}_{k}$ be the hyperplan defined by $x_{i}=0$
for $i=1,\ldots ,m$. Then $X\times_{k}Z_{i}$ are $T$-subvarieties
of~$\A^{m}_{X}$ and therefore, see Lemma~\ref{P-Lem}, we have
closed subschemes
$$
(X\times_{k}Z_{1})\times^{G}U\, ,\;\ldots\; ,\,
      (X\times_{k}Z_{m})\times^{G}U
$$
of $\A_{X}\times^{G}U$, where $(U,V)$ is a
$(\A^{m}_{X},T)$-admissible pair. Since $U\to U/T$
is a $T$-torsor (by assumption) we have
$$
\bigcap\limits_{j\not\in I}(X\times_{k}Z_{j})\times^{G}U\ =\,
   (X\times_{k}Z_{I})\times^{G}U
$$
for all $I\in\{ 1,\ldots ,m\}$, where
$Z_{I}=\bigcap\limits_{j\not\in I}Z_{j}$.
From Example~\ref{CKconeComplExpl2} we get then
a convergent spectral sequence
\begin{equation}
\label{1-EqSpSeqEq}
\tilde{E}^{p,q}_{1}\, =\;\bigoplus\limits_{|I|=p}
\HM^{T}_{-q-m}(X\times_{k}Z_{I},\CM_{n})\;\;\Longrightarrow\;
\HM^{T}_{-p-q}(X\times_{k}T,\CM_{n})
\end{equation}
for all cycle modules~$\CM_{\ast}$.

\smallbreak

By Example~\ref{G-cohTorsorEx} we have
$\HM^{T}_{-p-q}(X\times_{k}T,\CM_{n})\simeq
\HM_{-p-q-m}(X,\CM_{n+m})$ and since
$$
Z_{I}\, =\,\A^{|I|}_{k}=\Spec k[x_{i},i\in I]\,
\hookrightarrow\A^{m}_{k}
$$
is a $T$-equivariant vector bundle over the base
the pull-back
$$
\pi_{I\, T}^{\ast}\colon \HM^{T}_{-q-m-|I|}(X,\CM_{n+|I|})
\too\HM^{T}_{-q-m}(X\times_{k}Z_{I},\CM_{n})
$$
is an isomorphism, where $\pi_{I}\colon X\times_{k}Z_{I}\to X$
is the projection. Replacing~$q$ by $q+m$ the spectral
sequence takes therefore the following form
\begin{equation}
\label{2-EqSpSeqEq}
E^{p,q}_{1}\, =\;\bigoplus\limits_{|I|=p}
\HM^{T}_{-q-p}(X,\CM_{n+p})\;\;\Longrightarrow\;
\HM_{-p-q}(X,\CM_{n+m})\, .
\end{equation}

\smallbreak

We compute the differential
$$
d^{p,q}_{1}\colon \bigoplus\limits_{|I|=p}
\HM^{T}_{-q-p}(X,\CM_{n+p})\too
\bigoplus\limits_{|J|=p+1}\HM^{T}_{-q-p-1}(X,\CM_{n+p+1})\, .
$$
If $J\not\supseteq I$ the the $IJ$-component of~$d^{p,q}_{1}$
is zero. If $J\supset I$ let $J=\{ i_{1},\ldots ,i_{p+1}\}$
and $I=J\setminus\{ i_{r}\}$ for some $1\leq r\leq p+1$. Then
by Example~\ref{CKconeComplExpl2} the $IJ$-component
of the differential~$\tilde{d}^{p,q}_{1}$ of the spectral
sequence~(\ref{1-EqSpSeqEq}) is equal $(-1)^{r-1}$-times the
push-forward along the closed embedding (see Lemma~\ref{P-Lem})
$$
\iota_{IJ}\times^{T}\id_{U}\colon (X\times_{k}Z_{I})\times^{T}U\,
    \hookrightarrow\, (X\times_{k}Z_{J})\times^{T}U\, ,
$$
where $\iota_{IJ}$ is the closed immersion
$X\times_{k}Z_{I}\hookrightarrow X\times_{k}Z_{J}$.
Using the above identification we have then a commutative diagram
$$
\xymatrix{
E^{p,q}_{1}\, = & \HM_{-q-p}^{T}(X,\CM_{n+p})
       \ar[r]^-{\pi^{\ast}_{I\, T}}_-{\simeq} \ar[d]_-{d^{p,q}_{1}} &
               \HM_{-q}^{T}(X\times_{k}Z_{I},\CM_{n})
                       \ar[d]^-{(-1)^{r-1}\iota_{IJ,T\,\ast}}
\\
E^{p+1,q}_{1}\, = & \HM_{-q-p-1}^{T}(X,\CM_{n+p+1})
     \ar[r]^-{\pi^{\ast}_{J\, T}}_-{\simeq} &
               \HM_{-q}^{T}(X\times_{k}Z_{J},\CM_{n}) \rlap{\, ,}
}
$$
and therefore
$$
\begin{array}{r@{\; =\;}l}
   (-1)^{r-1}d^{p,q}_{1}(x) & (\pi^{\ast}_{J\, T})^{-1}
         \big[\iota_{IJ,T\,\ast}\big(\pi^{\ast}_{I\, T}(x)\big)\big] \\[2mm]
     & (\pi^{\ast}_{J\, T})^{-1}\Big[\iota_{IJ,T\,\ast}\Big(\one^{T}_{X\times Z_{I}}
             \cap\big( (\pi_{J}\times^{T}
                   \id_{U})\circ (\iota_{IJ}\times^{T}\id_{U})^{\ast}(x)\big)\Big)\Big] \\[2mm]
     & \big[(\pi_{J\, T}^{\ast})^{-1}(\iota_{IJ,T\,\ast}(\one^{T}_{X\times Z_{I}}))\big]
           \cap x \\[2mm]
     & \ch_{1}(L_{X}(\chi_{r}))\cap x
\end{array}
$$
for all $x\in\HM_{-q-p}^{T}(X,\CM_{n+p})$.
The last equation since $\iota_{IJ}$ is the zero section
of the pull-back of~$L_{X}(\chi)$ along the projection
$X\times_{k}Z_{I}\to X$.
\end{ntt}

\begin{ntt}[\bf Equivariant cycle cohomology]
From now on all varieties are assumed to be smooth.
Given a smooth equidimensional variety $X$ we define its
$i$-th {\it cycle cohomology group} as
$$
\HM^{i}(X,\CM_{n})\, :=\;\HM_{\dim X -i}(X,\CM_{n-\dim X})\, .
$$
A pairing
$$
\MK{i}(E)\,\times\,\CM_{j}(E)\too\CM_{i+j}(E)
$$
induces the so called {\it intersection product} of cycle
cohomology groups
$$
\HM^{i}(X,\MK{m})\,\times\,\HM^{j}(X,\CM_{n})\;\too\,
   \HM^{i+j}(X,\CM_{m+n})\, ,\;\; (\alpha ,x)\,\longmapsto\,\alpha\cap x\, ,
$$
see~\cite[Sect.\ 14]{Ro96}. This generalizes the usual intersection
product of Chow groups as defined for instance in the book~\cite{IT}
of Fulton.
\end{ntt}

\begin{dfn}\label{keydef}
It follows by descent (since $U\too U/G$ is a principal bundle)
that $X\times^{G}U$ is an equidimensional smooth variety, too.
Hence we can define
$$
\HM^{i}_{G}(X,\CM_{n})\, :=\HM_{\dim X-i}^{G}(X,\CM_{n-\dim X})
 \, =\,\HM^{i}(X\times^{G}U,\CM_{n})
$$
and call it the $i$-th {\it $G$-equivariant cycle cohomology group}
of~$X$ with values in $\CM_{\ast}$. 
\end{dfn}

The pairing of
$\HM^{\ast}(X\times^{G}U,\MK{\ast})$ with $\HM^{\ast}(X\times^{G}U,\CM_{\ast})$
induces a pairing
$$
\HM^{i}_{G}(X,\MK{m})\,\times\,\HM^{j}_{G}(X,\CM_{n})\too
       \HM^{i+j}_{G}(X,\CM_{m+n})\, ,\;\, (\alpha,x)\,\longmapsto\, \alpha\cap x\, .
$$
Since the pull-back of ordinary cycle cohomology groups respects the
product, we can use Bogomolov's double filtration argument
to check that this definition does not depend (up to natural isomorphism)
on the choice of an $(X,G)$-admissible pair.

\smallbreak

It follows from the properties of the ordinary ``intersection''
product in cycle cohomology that~$\HM^{\ast}_{G}(X,\MK{\ast})$
becomes a skew-commutative ring with this product, \ie
$$
\alpha\cap\beta\, =\, (-1)^{(i+m)\cdot (j+n)}\cdot\beta\cap\alpha
$$
for all $\alpha\in\HM_{G}^{i}(X,\MK{m})$ and
$\beta\in\HM_{G}^{j}(X,\MK{n})$, \cf~\cite[Prop.\ 56.4]{AGQ}.
We denote the neutral element of this multiplication
by~$\one^{G}_{X}$. This is the rational equivalence class
of $X\times^{G}U$ in $\HM^{0}_{G}(X,\MK{0})=\HM^{0}(X\times^{G}U,\MK{0})$.

\smallbreak

\begin{ntt}[Projection formulas]
From the properties of the ordinary intersection
product shown in~\cite[Prop.\ 5.9]{Deg06} we obtain
the following: 

\smallbreak

Let $f\colon X\too Y$ be a morphism of equidimensional smooth $G$-varieties. Then
\begin{itemize}
\item[(i)]
$f^{\ast}_{G}(\beta\cup y)\, =\, f^{\ast}_{G}(\beta)\cap f^{\ast}_{G}(y)$, and

\smallbreak

\item[(ii)]
if $f$ is proper the equations (projection formulas)
$$
f_{G\ast}(\alpha\cap f^{\ast}_{G}(y))\, =\, f_{G\ast}(\alpha)\cap y\;\;\;\mbox{and}\;\;\;
       f_{G\ast}(f^{\ast}_{G}(\beta)\cap x)\, =\,\beta\cap f_{G\ast}(x)
$$
\end{itemize}
for all $\alpha\in\HM^{i}_{G}(X,\MK{m})$, $\beta\in\HM^{i}_{G}(Y,\MK{m})$,
$x\in\HM^{j}_{G}(X,\CM_{n})$, and $y\in\HM^{j}_{G}(Y,\CM_{n})$.

\smallbreak

We denote by $\ch_{1}(L)$
also the element $\ch_{1}(L)(\one^{G}_{X})\in\CH^{1}_{G}(X)$.
With this notation we have by the projection formula
$$
\ch_{1}(L)(x)\, =\,\ch_{1}(L)\cap x
$$
for all $x\in\HM^{i}_{G}(X,\CM_{n})$ and all $i\in\N$
and~$n\in\zz$ if the variety~$X$ is smooth $k$.
\end{ntt}

This section can be summarized by the following  (see~\ref{keydef} and \ref{axioms})

\begin{thm}
The functor $X\mapsto\HM^{\ast}_{G}(X,\CM_{\ast})$ provides an example of
a graded essential $G$-equivariant pretheory. 
\end{thm}

\begin{ex}[Equivariant Chow groups]\label{equivChow}
We have
$$
\HM_{i}^{G}(X,\MK{-i})\, =\,\CH^{G}_{i}(X)\, ,
$$
where $\CH^{G}_{\ast}$ denotes the $G$-equivariant Chow-theory
of Edidin and Graham~\cite{EdGr98}.

\smallbreak

Since $\HM_{j}(Y,\MK{-i})=0$ for $j\leq i-1$ we know by the
localization sequence that for a $G$-equivariant open embedding
$\iota\colon W\hookrightarrow X$ the pull-back
$$
\iota_{G}^{\ast}\colon \CH^{G}_{i}(X)\too\CH^{G}_{i}(W)
$$
is surjective. Identifying $\CH^i_G(X)$ with $\CH^G_{\dim X-i}(X)$ we obtain that the functor
$$
\hh_G\colon X\,\longmapsto\,\CH^{\ast}_{G}(X)\, =\,
   \bigoplus\limits_{i\in\zz}\CH^{i}_{G}(X)
$$
provides an example of an essential $G$-equivariant pretheory.
\end{ex}

%%%%%%%%%%%%%%%%%%%%%%%%%%%
%%%%%%%%%%%%%%%%%%%%%%%%%%%

\section{Torsors and equivariant maps}

In the present section we
generalize the result of Karpenko and Merkurjev~\cite[Thm.\ 6.4]{KaMe06} to
an arbitrary equivariant pretheory. Our arguments follow closely the exposition of \cite[\S6]{KaMe06}.

\smallbreak

Let $S=\GL(V)$ be the group of 
automorphisms of a finite dimensional $k$-vector space~$V$.
Let $H$ be an algebraic subgroup of $S$. 
Consider $S$ as a (left) $H$-variety.

\smallbreak

Let $\hh_H$ be a $H$-equivariant pretheory over $k$.
Following the proof of \cite[Prop.\ 6.2]{KaMe06} 
we embed $S$ into the affine space $\End_k(V)$ as a 
$S$-equivariant (and, hence, $H$-equivariant) open subset. 

\smallbreak

Let $\phi\colon S\to\pt$ denote the structure map. 
The induced pull-back $\phi_H^{*}$ factors as the
composite of pull-backs
$$
\hh_{H}(\pt) \stackrel{\simeq} \too \hh_{H}(\End(V)) \twoheadrightarrow \hh_{H}(S),
$$
where the first map is an isomorphism by homotopy invariance
and the second map is surjective by the localization property. This
proves that

\begin{lem}\label{varphiLem} The induced pull-back $\phi_H^*$ is surjective.
\end{lem}

Let $\mu_{s}\colon S\to S$ denote the right multiplication
by $s\in S(k)$.
Since $\phi\circ \mu_{s}=\varphi$ as morphisms over $k$ and $\mu_{s}$ is
$H$-equivariant, 
we have $(\mu_{s})_{H}^{*}\circ \phi^{*}_{H}=\phi^{*}_{H}$.
Since $\phi_{H}^{*}$
is surjective by Lemma~\ref{varphiLem}, this proves that

\begin{lem}\label{muLem}
The induced pull-back 
$(\mu_{s})_{H}^{*}\colon \hh_{H}(S)\to \hh_{H}(S)$
is the identity.
\end{lem}

%%%%%%%%%%%%%%%%%%%%%%%%%%%%%%%%%%%%%%%%%%

Let $G$ be an algebraic subgroup of $S$ such that $H\subseteq G\subseteq S$ 
so that $G$ is considered as a (left) $H$-variety.
Let $E$ be a (left) $G$-variety over $k$ and 
let $\eta_E\colon \Spec K\to E$ denote its generic point, where $K=k(E)$.

\smallbreak

Consider the $G$-equivariant (and, hence, $H$-equivariant) map
$$
\psi_E\colon G_K \,=\, G\times_{\Spec k}\Spec K
     \stackrel{(\id,\eta_E)}\too G\times_{\Spec k} E \too E
$$
which takes the identity of $G$ to the generic point of $E$.
Suppose that there is a
$G$-equivariant map $\rho\colon E\to S$ over $k$.
Then there is a commutative diagram of $H$-equivariant maps
$$
\xymatrix{
G_K\ar[r]^{\psi_E}\ar[d]_i & E \ar[r]^\rho & S \\
S_K \ar[rr]^{\mu_{\rho(\eta_E)}} & & S_K \ar[u]_p
}
$$
where the map $i$ is the embedding, $p$
is the projection $S_K\,=\, S\times_{\Spec k}\Spec K \to S$ to the first factor and
the bottom horizontal map is the multiplication by $\rho(\eta_E)$.

By the diagram the pull-back of the composite
$(\psi_E)_H^*\circ \rho_H^*=(\rho\circ \psi_E)^*_H$ coincides
with the pull back $(p\circ \mu_{\rho(\eta_E)}\circ i)^*_H$. Here the map $(\psi_E)^*_H\colon \hh_H(E)\to \bar\hh_H(G_K)$
is the canonical map to the colimit. 
By Lemma~\ref{muLem} the latter
coincides with the pull-back $i_H^*\circ p_H^*$, hence, proving the following

\begin{lem}
\label{factorLem} Let $E$ be a $G$-variety together with a $G$-equivariant map $\rho\colon E\to S$. Then we have
$$
(\psi_E)_H^*\circ \rho_H^*=i_H^*\circ p_H^*\colon \hh_H(S) \to \bar\hh_H(G_K).
$$
\end{lem}

We are now in position to prove the main result of
this section

\begin{thm}
\label{mainThm}
Let $H\subset G$ be algebraic groups and let $\hh_{H}(-)$
be a $H$-equivariant pretheory.
Then for any $G$-torsor $E$ with $K=k(E)$
we have
$$\Image (\varphi_H^*)\subseteq \Image((\psi_E)_H^*) \text{ in }\bar\hh_{H}(G_K),$$
where $\varphi \colon G_K\to \pt$ is the structure map.
\end{thm}

\begin{proof}
By Lemma~\ref{varphiLem} we have 
$$
\Image(\varphi_H^*)=\Image(i^*_H\circ p^*_H\circ \phi_H^*)=\Image(i^*_H\circ p^*_H).
$$
Theorem then follows from Lemma~\ref{factorLem}
and the fact that there exists a finite dimensional
$k$-vector space~$V$ and a $G$-equivariant map $E\too S=\GL (V)$,
see \cite[Prop.~6.4]{KaMe06}. 
\end{proof}

\begin{dfn}\label{geninv}
Let $H\subset G$ be algebraic groups over $k$ and $\hh_{H}(-)$
be an equivariant pretheory with values in the category of commutative rings. 
To each $G$-torsor $E$
we associate a commutative ring
$$
E\mapsto \bar\hh_{H}(G_K)\otimes_{\hh_{H}(E)}\bar\hh_{H}(H_K),
$$
where $\bar\hh_{H}(G_K)$ is the $\hh_H(E)$-module via $(\psi_E)_H^*$ and $\bar\hh_H(H_K)$ is
the $\hh_H(E)$-module via the composite $\hh_H(E)\stackrel{(\psi_E)^*_H}\to \bar\hh_H(G_K)\to \bar\hh_H(H_K)$ with the last map induced by the embedding 
$H\subset G$.

\smallbreak

This ring will be denoted by $\widehat\hh_{H}(E)$ and
will play the central role in the last section of this paper. In particular, it will be shown
that for known examples of equivariant pretheories it is always a quotient of the cohomology ring  $\hh(G)$ of $G$.
\end{dfn}

\begin{cor}\label{mainThm2} Let $H\subset G$ be algebraic groups over $k$ and let $\hh_H(-)$
be an essential $H$-equivariant pretheory. Then
there exists a field extension $l/k$ and a $G$-torsor $E$ over $l$ with $L=l(E)$ such that
$$\Image (\varphi_{H_l}^*)= \Image((\psi_E)_{H_l}^*) \text{ in }\bar\hh_{H_l}(G_L)$$
\end{cor}

\begin{proof}
We fix
an embedding $G\too S=\GL (V)$ for some finite dimensional
$k$-vector space~$V$. The quotient $S\too G\backslash S$ (for
the right action of~$G$ on~$S$) is a (left) $G$-torsor. Let~$l$
be its function field and consider the cartesian square
$$
\xymatrix{
E \ar[r]^{\rho} \ar[d] & S \ar[d]
\\
\Spec l \ar[r] & G\backslash S\, .
}
$$
Since $S\too G\backslash S$ is a $G$-torsor the map $E\too\Spec l$
is a $G$-torsor, too. The $l$-scheme $E$ is a localization of~$S$
and, therefore, by (C) and localization property (L) the pull-back
$\rho_{H}^{\ast}\colon\hh_{H}(S)\to\bar\hh_{H}(E)\to \hh_H(E)$
is surjective.  This implies that the pull-back $\rho_{H_l}^*\colon\hh_{H_l}(S_l)\to \hh_{H_l}(E)$
is surjective.
It remains to apply the proof of Theorem~\ref{mainThm} over $l$ and to observe
that 
$\Image(\varphi_{H_l}^*)=\Image((\psi_E)^*_{H_l}\circ \rho^*_{H_l})=\Image((\psi_E)^*_{H_l})$.
\end{proof}

\begin{dfn}
A $G$-torsor~$E$ over $k$ which satisfies the equality of Corollary~\ref{mainThm2} 
will be called a {\it generic torsor} with
respect to the pretheory $\hh_{H}(-)$.
Note that in the proof
we provided an example of a $G$-torsor which is generic over some field extension for
all equivariant pretheories.
\end{dfn}

\begin{ex}
Observe that for a generic $G$-torsor $E$ over $k$ we have 
$$
\widehat{\hh}_{H}(E)=\bar{\hh}_{H}(G_K)
   \otimes_{\hh_{H}(\pt)}\bar\hh_{H}(H_K),
$$
where $\bar{\hh}_{H}(G_K)$ is an
$\hh_{H}(\pt)$-module via $\varphi^{*}_{H}$,
and for the trivial $G$-torsor $G$ we have
$$
\widehat{\hh}_{H}(E)=\bar\hh_{H}(G_K)\otimes_{\hh_H(G)} \bar\hh_H(H_K).
$$
\end{ex}

%%%%%%%%%%%%%%%%%%%%%%%%%%%%%%%%%%%%%%%
%%%%%%%%%%%%%%%%%%%%%%%%%%%%%%%%%%%%%%%
%%%%%%%%%%%%%%%%%%%%%%%%%%%%%%%%%%%%%%%

\section{Equivariant oriented cohomology}
\label{ApplExplSect}

In the present section we apply Theorem~\ref{mainThm} to the case
of a $B$-equivariant oriented cohomology, where $B$ is a Borel subgroup of a split
semisimple linear algebraic group.

\smallbreak

Let $G$ be a split semisimple linear algebraic group of rank $n$ over a field $k$ and
let $T$ be a split maximal torus of $G$.
Similarly to \ref{eqspecsec}
consider the action \eqref{actT} of $T$ on the affine space $\A^{n}_{k}$
with weights $\chi_{1},\ldots,\chi_{n}$ together with an action of~$T$
on~$G$ by left multiplication.
Then $T$ embeds into $\A^{n}_{k}=\Spec k[x_{1},\ldots ,x_{n}]$
as the complement of the coordinates hyperplanes $Z_{i}$, $i=1,\ldots ,n$. 
Let $V=\A^{n}_{k}\times^{T} G$ be the associated vector
bundle over~$G/T$ (see \cite[p.22]{Br98}). By definition
$V=L_{G/T}(\chi_{1})\oplus\ldots \oplus L_{G/T}(\chi_{n})$
and $G=T\times^{T} G$  embeds into~$V$ as the complement of the union
of zero-sections
$$
V_{j}\, =\,\bigoplus\limits_{j\neq i}L_{G/T}(\chi_{i})=Z_{j}\times^{T} G\, ,
\quad i=1,\ldots ,n\, .
$$ 
Note that $e_{j}\colon V_{j}\hookrightarrow V$ is a smooth
subvariety for every~$j$.

\medbreak

Let now $\hh(-)$ be an oriented cohomology theory in the
sense of~\cite{AC}, i.e. a~contravariant functor from the category of smooth varieties over $k$ to the category of graded commutative rings satisfying certain axioms. In particular, if~$X$ is a $k$-variety
with an open subvariety $\iota\colon U\hookrightarrow X$ there is
an exacts sequence
$$
\xymatrix{
\hh (Z) \ar[r]^-{j_{\ast}} & \hh (X) \ar[r]^-{\iota^{\ast}} &
               \hh (U) \ar[r] & 0
}
$$
where $j\colon Z=X\setminus U\hookrightarrow X$ is the closed
complement of~$U$, and there is also a first Chern class
which we denote by~$c_{1}^{\hh}$.

\smallbreak

Having such a theory~$\hh (-)$ we get from this localization
sequence (by induction) an exact sequence
$$
\bigoplus_{j=1}^{n} \hh(V_{j}) \stackrel{\oplus_{j}
      (e_{j})_{*}}\too \hh(V) \to \hh(G) \to 0
$$
By the properties of the first Chern class we have
$$(e_{j})_{*}(1_{V_{j}})=c_{1}^\hh(L_{V}(\chi_{j}))$$
which implies that the image of $\oplus_{j}(e_{j})_{*}$ is an
ideal generated by the first Chern classes
$c_{1}^\hh(L_{V}(\chi_{j}))$, $j=1,\ldots ,n$.

\smallbreak

Let $B$ be a Borel subgroup of $G$ containing $T$ and let $G/B$ be the variety of Borel subgroups.
By \cite[Thm.10.6 of Ch.III]{Bo69} the composite of projections $V\to G/T \to G/B$ is a chain of affine bundles.
Therefore, by the homotopy invariance there is an isomorphism 
$\hh(G/B)\stackrel{\simeq}\to \hh(V)$ compatible with the Chern classes and we obtain the following

\begin{prop}\label{Giso}
There is an isomorphism of rings
$$
\hh(G)\simeq \hh(G/B)/\big(c_{1}^\hh(L_{G/B}(\chi_{1})),
          \ldots, c_{1}^\hh(L_{G/B}(\chi_{n}))\big)\, ,
$$
where $\chi_{1},\ldots ,\chi_{n}$ is a basis of the
character group~$T^{*}$.
\end{prop}

Let $\hh_B(-)$ be an $B$-equivariant pretheory 
to the category of commutative rings such that
\begin{itemize}
\item[(i)]
$\hh_{B}(E)=\hh(E/B)$ for every $G$-torsor $E$, where $\hh(-)$ is an oriented cohomology
in the sense of \cite{AC}.

\smallbreak

\item[(ii)] $\bar\hh_{B}(B_K)=\hh(\pt)$ and $\bar\hh_{B}(G_K)\simeq \hh(G/B)$. 
\end{itemize}

\smallbreak

Then
 the ring $\hat\hh_B(E)=\bar{\hh}_{B}(G_K)\otimes_{\hh_{B}(E)}\bar{\hh}_{B}(B_K)$ introduced in Definition~\ref{geninv} can be identified
with a quotient of $\bar\hh_B(G_K)\simeq \hh(G/B)$ modulo the ideal generated by non-constant elements from the image of the restriction
$(\psi_E)_B^{\ast}\colon \hh(E/B)\to \hh(G/B)$.

\smallbreak

Consider now the map 
$\varphi_{B}^{*}\colon \hh_{B}(\pt)
   \too\bar{\hh}_{B}(G_K)\simeq \hh(G/B)$.
By Theorem~\ref{mainThm} $\Image(\varphi^*_B)\subseteq \Image((\psi_E)^*_B)$, hence,
$\hat\hh_B(E)$ can be identified with a quotient of the factor ring $\hh(G/B)/I$, where $I$ denotes
the ideal generated by elements from the image of $\varphi^*_B$ which are
in the kernel of the augmentation.

\smallbreak

Then by Proposition~\ref{Giso} and Corollary~\ref{mainThm2} we obtain
the following

\begin{cor}\label{corGiso} 
Assume that the image of $\varphi_{B}^{\ast}$
is generated by the Chern classes $c_1^\hh(L_{G/B}(\chi_i))$
of line bundles associated to the characters $\chi_i\in T^*$  ($i=1\ldots n$). 
Then  $\hat\hh_B(E)$ is a quotient of $\hh(G/B)/I\simeq \hh(G)$.
Moreover, if $\hh_B(-)$ is essential and $E$ is generic, then
$$
 \hat\hh_B(E)\simeq\hh(G).
$$
\end{cor}

\begin{ex}[Chow groups and the $J$-invariant]
Consider the equivariant Chow groups $\hh_{B}(-)=\CH^{B}(-)$. Let $E$ be a $G$-torsor.
The ring $\hh_{B}(\pt)$ can be identified with the
symmetric algebra $S(T^{*})$ and the map 
$$
\varphi_{B}^{*}\colon S(T^{*})=\hh_{B}(\pt)
   \too\bar{\hh}_{B}(G_K)=\CH(G/B)
$$
coincides with the characteristic map for Chow
groups  studied in~\cite{De74}. So its image is
generated by the first Chern classes $c_{1}(L_{G/B}(\chi_{i}))$
of the respective line bundles.

\smallbreak

The map
$(\psi_E)_B^{*}$
coincides with the restriction map
$$
\res\colon \CH(E/B) \too \CH(G/B),
$$
where $E/B$ is the twisted form of $G/B$ by means of $E$
and the map $$S(T^{*})=\hh_{B}(\pt)\too \hh_{B}(B_K)=\CH(\pt)=\zz$$
is the augmentation map. If $E$ is generic, then we have
$$
\widehat \hh_{B}(E)\simeq \CH(G/B)\otimes_{S(T^{*})}
           \zz\simeq \CH(G).
$$
where the last isomorphism follows by Corollary~\ref{corGiso}.
For an arbitrary $G$-torsor~$E$ the ring
$$
\widehat \hh_{B}(E)=\CH(G/B)\otimes_{\Image(\res)}\zz
$$
is a quotient ring of $\CH(G/B)$ modulo the ideal generated
by non-constant elements from the image of the restriction $\CH(E/B)\to \CH(G/B)$.

\smallbreak

Observe that the characteristic map $\varphi^{*}_{B}$
is not surjective in general. However, its image is a subgroup of finite index in 
$\CH(G/B)$ measured by the torsion index $\tau$ of $G$.
This implies that  for a~$G$-torsor~$E$ we have
$$\widehat\hh_{B}(E)\otimes_{\zz}\mathbb{Q}\simeq\mathbb{Q}.$$

If $p\mid\tau$, then there is
an isomorphism
$$
\widehat{\hh}_{B}(E)\otimes_{\zz}\zz/p\simeq 
\frac{\zz/p\,[x_1,\ldots,x_r]}{(x_{1}^{p^{j_{1}}},
          \ldots,x_{r}^{p^{j_{r}}})},
$$
where $(j_{1},\ldots,j_{r})$ is the {\it $J$-invariant}
of~$G$ twisted by $E$  as defined in \cite{PeSeZa08}.
Observe that $j_{i}\leq k_{i}$, $i=1\ldots r$,
where $k_{i}$ are defined via the
{\em $p$-exceptional degrees} introduced by Kac~\cite{Kc85},
and for a generic torsor~$E$ we have equalities $j_{i}=k_{i}$
for each $i$.
\end{ex}

\begin{ex}[Grothendieck's $K_{0}$ and indexes of the Tits algebras]

Consider the equivariant $K_{0}$-groups $\hh_{B}(-)=\QK_{0}(B,-)$. Let $E$ be a $G$-torsor. 
The ring
$\hh_{B}(\pt)$ can be identified with the integral
group ring $\zz[T^{*}]$ and with the representation ring $\Rep T$ of $T$, i.e.
$$
\hh_{B}(\pt)=\zz[T^{*}]=\Rep T.
$$
The map
$$
\varphi_{B}^{*}\colon \zz[T^{*}]=\hh_{B}(\pt)
      \too\bar\hh_{B}(G_K)\simeq \QK_{0}(G/B)
$$
coincides with the characteristic map $\cc$ for $K_{0}$
studied in \cite{De74} and again its image is generated by the first Chern classes.

\smallbreak

As before the map
$(\psi_E)_B^*$
coincides with the restriction map
$$
\res\colon \QK_0(E/B) \too \QK_0(G/B),
$$
and applying \ref{mainThm}
we obtain the following $K_0$-analogue of the
Karpenko-Merkurjev result:

\begin{cor}
\label{ImageResThm}
Let $E$ be a $G$-torsor over $k$
and let $E/B$ be a twisted form of~$G/B$ by $E$.
Then
\begin{itemize}
\item[(i)]
$\cc(\zz[T^{*}])\,\subseteq\, \res(\QK_{0}(E/B))$;

\smallbreak

\item[(ii)]
there exists a $G$-torsor~$E$ over some field extension of $k$ such that 
$$\cc(\zz[T^{*}])=\res(\QK_0(E/B)).$$
\end{itemize}
\end{cor}

According to a result of Panin~\cite{Pa94}
the image of the restriction map is given by the sublattice
$$
\{i_{w,E} \cdot g_{w}\}_{w\in W},
$$
where $W$ is the Weyl group of~$G$,
$\{g_{w}\}_{w\in W}$ is the Steinberg basis
of $\QK_{0}(G/B)$ and $\{i_{w,E}\}$ are 
indexes of the respective Tits algebras.

\smallbreak

Corollary~\ref{ImageResThm} implies that there exists a maximal set
of indexes $\{m_{w}\}_{w\in W}$ such that
\begin{itemize}
\item
$i_{w,E}\leq m_{w}$ for every $w\in W$ and
every torsor $E$;

\smallbreak

\item
there exists~$E$ such that $i_{w,E}=m_{w}$
for every $w\in W$;

\smallbreak

\item
the image of the characteristic map
$\varphi^{*}_B(\zz[T^{*}])$
coincides with the sublattice
$\{m_{w}\cdot g_{w}\}_{w\in W}$, hence,
providing a way to compute $m_{w}$.
\end{itemize}

\smallbreak

The indexes $m_{w}$ are called the {\it maximal Tits indexes}.
They have been extensively studied by
Merkurjev~\cite{Me96} and Merkurjev, Panin and
Wadsworth~\cite{MePaWa96}. They
are closely related to the dimensions of
irreducible representations of~$G$. 
Comparing with the case of Chow groups 
one observes that the maximal Tits indexes
in~$\QK_{0}$ play the same role as
the $p$-exceptional degrees in Chow groups. 

\smallbreak

Since the map $\zz[T^{*}]=\hh_{B}(\pt)
\too\bar\hh_{B}(B_K)=\QK_{0}(\pt)=\zz$
is the augmentation map, for a generic torsor~$E$
we have
$$
\widehat \hh_{B}(E)=\QK_{0}(G/B)\otimes_{\zz[T^{*}]}
       \zz\simeq \QK_{0}(G),
$$
where the last isomorphism follows by Corollary~\ref{corGiso}.
Hence, for an arbitrary $G$-torsor~$E$
$$
\widehat \hh_{B}(E)=\QK_{0}(G/B)\otimes_{\Image (\res)}\zz
$$
is the quotient ring of $\QK_{0}(G/B)$ modulo the ideal generated
by elements from the image of the restriction $\QK_0(E/B)\to \QK_0(G/B)$ which are in the kernel of the augmentation.
\end{ex}

\begin{ex}[Equivariant algebraic cobordism]

Consider the equivariant algebraic cobordism $\hh_{B}(-)=\Cob^{B}(-)$. Let $E$ be a $G$-torsor.
The completion $\hh_B(\pt)^\wedge$ 
of $\hh_{B}(\pt)$ at the augmentation ideal, 
(the kernel of $\hh_B(pt)\to \hh_B(B)$)
can be identified with the formal group ring $\mathbb{L}[[T^{*}]]_U$ 
introduced in \cite[Def.~2.4 and 2.7]{CPZ10}, 
where $\mathbb{L}$ is the Lazard ring
and $U$ denotes the universal formal group law.

\smallbreak

The map 
$$
\varphi_{B}^{*}\colon \mathbb{L}[[T^*]]_U=\hh_{B}(\pt)^{\wedge}
   \too\bar{\hh}_{B}(G_K)=\Cob(G/B)
$$
coincides with the characteristic map of~\cite[Def.~10.2]{CPZ10}
and its image is generated by the first Chern classes. 

\smallbreak

The map
$(\psi_E)_B^{*}$
coincides with the restriction map
$$
\res\colon \Cob(E/B) \too \Cob(G/B),
$$
where $E/B$ is the twisted form of $G/B$ by means of $E$
and the map $\mathbb{L}[[T^{*}]]_{U}=\hh_{B}(\pt)^{\wedge}\too 
\hh_{B}(B_K)=\Cob(\pt)=\mathbb{L}$
is the augmentation map.
By Corollary~\ref{corGiso} for an arbitrary $G$-torsor $E$
we have
$$
\widehat \hh_{B}(E)=\Cob(G/B)\otimes_{\Image(\res)}\mathbb{L}.
$$
is a quotient of the ring $\Cob(G/B)$ modulo the image of the restriction
$\Cob(E/B)\to \Cob(G/B)$ from the kernel of the augmentation.
\end{ex}

%%%%%%%%%%%%%%%%%%%%%%%%%%%%%%%%%%%%%%%
%%%%%%%%%%%%%%%%%%%%%%%%%%%%%%%%%%%%%%%

\begin{appendix}
\section{A spectral sequence}
\label{SpSeqSect}

\noindent
Let~$X$ be a $k$-scheme, $Z_{1},\ldots ,Z_{m}$ closed
subschemes of~$X$ (with reduced structure), and $\CM_{\ast}$ a cycle
module over~$k$. We construct in this section a spectral sequence which
converges to the cycle homology of~$\CM_{\ast}$ over the open complement
of $\bigcup\limits_{i=1}^{m}Z_{i}$ in~$X$ generalizing the long exact
localization sequence (the case $m=1$). A reader familiar with
Levine~\cite[Sect.\ 1]{Le93} will notice an analogy with Levine's
construction of a similar spectral sequence for Quillen $K$-theory
which should have an explanation in the theory of model categories.

\bigbreak

Let~$m$ be a positive integer and~$\ul{m}$ the set of subsets
of $\{ 1,2,\ldots ,m\}\subset\N $. We set further $\ul{0}=\emptyset$.
We consider~$\ul{m}$ as a category with $\Mor_{\ul{m}}(I,J)=\{\emptyset\}$
(one element set) if $I\subseteq J$ and~$\emptyset$ otherwise.
\begin{dfn}
\label{Defm-cube}
A {\it $m$-cube of complexes} is a functor
$$
K_{\ast}\colon \ul{m}\too\Kb(\Ab)\qquad
      I\,\longmapsto\, K_{I}
$$
from~$\ul{m}$ into the category of bounded (homological)
complexes of abelian groups. If $I\subseteq J$ we denote
$r_{IJ}^{K}$ the induced morphism $K_{I}\too K_{J}$,
and set $r^{K}_{IJ}=0$ if $I\not\subseteq J$.
We observe that $r^{K}_{JL}\cdot r^{K}_{IJ}=r^{K}_{IL}$
if $I\subseteq J\subseteq K$ since $I\mapsto K_{I}$ is
a functor.

\smallbreak

A {\it morphism} of $m$-cubes is a natural transformation.
\end{dfn}

For brevity of notation we define if $l=|J|=|I|+1$ a morphism~$\ep^{K}_{IJ}$
as follows: If $I\not\subset J$ we set $\ep^{K}_{IJ}=0$ and if
$J=\{ i_{1}<i_{2}<\ldots <i_{l}\}$ and $I=J\setminus\{ i_{d}\}$ for
some $1\leq d\leq l$ we set $\ep^{K}_{IJ}=(-1)^{d-1}\cdot r_{IJ}^{K}$.
The signs are chosen, such that the matrix product
\begin{equation}
\label{matrix0Eq}
\big(\ep_{JL}\big)_{|J|=p+1,|L|=p+2}\;\cdot\;
      \big(\ep_{IJ}\big)_{|I|=p,|J|=p+1}
\end{equation}
is zero.

\medbreak

With a $m$-cube of complexes~$K_{\ast}$ we can associate two $(m-1)$-cubes
of complexes (as long as $m\geq 2$). First we have the $(m-1)$-cubes
of complexes $\tilde{K}_{\ast}$ which is the restriction of~$K_{\ast}$
to $\ul{m-1}$, and second we have $K'_{\ast}$. The latter is defined as
the composition of $\ul{m-1}\too\ul{m}$, $J\longmapsto J\cup\{ m\}$,
with $K_{\ast}$, \ie
$$
K'_{\ast}\colon \ul{m-1}\too\Db(\Ab)\qquad
  J\,\longmapsto\, K_{J\cup\{ m\}}\, .
$$

\begin{ex}\label{CKconeComplExpl1}
Let~$X$ be a $k$-scheme with closed subschemes
$Z_{1},\ldots ,Z_{m}$, and $\CM_{\ast}$ a cycle module.
We set $Z_{I}:=\bigcap\limits_{j\not\in I}Z_{j}$ for all~$I$
in~$\ul{m}$. We have than an $m$-cube of complexes
$$
K_{\ast}\colon \CK_{\smb}(X,Z_{1},\ldots ,Z_{m},\CM_{n},\ast)\colon 
I\,\longmapsto\,\CK_{\smb}(Z_{I},\CM_{n})\, ,
$$
where for $I\subseteq J$ the morphism $r_{IJ}^{\CK}$
is the push-forward along the inclusion of closed subschemes
$\bigcap\limits_{j\not\in I}Z_{j} \hookrightarrow
\bigcap\limits_{j\not\in J}Z_{j}$. Then we have
$$
K'_{\ast}\, =\,\CK_{\smb}(X,Z_{1},\ldots ,Z_{m-1},\CM_{n},\ast)
$$
and
$$
\tilde{K}_{\ast}\, =\,\CK_{\smb}(Z_{m},Z_{1}\cap Z_{m}\,
         \ldots ,Z_{m-1}\cap Z_{m},\CM_{n},\ast)\, .
$$
Note also that
$K_{\emptyset}=\CK_{\smb}(\bigcap\limits_{i=1}^{m}Z_{i},\CM_{n})$
and $K_{\{ 1,\ldots ,m\}}=\CK_{\smb}(X,\CM_{n})$.
\end{ex}

\bigbreak

\noindent
\begin{ntt}[{\bf The functors~$\cfb_{\geq i}$}]
\label{Cfb-FunctorSubSect}
Recall first the cone of a morphism of complexes
$f\colon K_{\smb}\too L_{\smb}$. This is the complex
$\cone f$ which is given in degree~$i+1$ and~$i$ by:
$$
\xymatrix{
\ldots \ar[r] & K_{i}\oplus L_{i+1}
  \ar[rrr]^-{\left(\begin{array}{c@{\,\,}c} d^{K}_{i} & f_{i} \\[1mm] 
                         0 & -d^{L}_{i+1}\end{array}\right)}
      & & & K_{i-1}\oplus L_{i} \ar[r] & \ldots\, .
}
$$
We have then a (so called) exact triangle
$K_{\smb}\xrightarrow{f}L_{\smb}\too\cone f\too K_{\smb}[1]$
(with obvious morphisms on the right).

\smallbreak

We define inductively a functor~$\cfb_{\geq i}$ form the
category of $m$-cubes of complexes to $\Kb(\Ab)$ for all $i\in\zz$.
Let $K_{\ast}$ be such a $m$-cube. Then we set $\cfb_{\geq i}K_{\ast}=0$
if $i\geq m+1$ and
$$
\cfb_{\geq m}K_{\ast}\, :=\; K_{\{ 1,\ldots ,m\}}\, .
$$
We have then the morphism of complexes
$$
\Theta^{K}_{m}\, :=\;\sum\limits_{i=1}^{m}\ep_{I\{ 1,\ldots ,m\}}\colon 
  \bigoplus\limits_{|I|=m-1}K_{I}\too
      K_{\{ 1,\ldots ,m\}}\; =\,\cfb_{\geq m} K_{\ast} 
$$
and define $\cfb_{\geq m-1}K_{\ast}$ to be the cone of this morphism.

\smallbreak

The composition
$$
\bigoplus\limits_{|I|=m-2}K_{I}\,
\xrightarrow{\;\big(\ep_{IJ}\big)_{I,J}\;}
\bigoplus\limits_{|J|=m-1}K_{J}\,
\xrightarrow{\Theta^{K}_{m}\;}\, K_{\{ 1,\ldots ,m\}}
$$
is the zero morphism, see~(\ref{matrix0Eq}),
and therefore induces a morphism of complexes
$$
\Theta_{m-1}^{K}\colon \bigoplus\limits_{|I|=m-2}K_{I}[1]
\too\cfb_{\geq m-1}K_{\ast}
    \, =\,\cone\Theta^{K}_{m}\, .
$$
We set then $\cfb_{\geq m-2}\, :=\,\cone\Theta^{K}_{m-1}$. Let now $p\leq m-3$.
Then by (descending) induction we have a morphism of complexes
$$
\Theta_{p+2}^{K}\colon \bigoplus\limits_{|L|=p+1}K_{L}[m-p-2]\too
   \cfb_{\geq p+2}K_{\ast}
$$
such that we have
$\Theta_{p+2}^{K}\,\cdot\,\big(\ep_{JL}[m-p-2]\big)_{|J|=p,|L|=p+1}\, =0$.
Therefore there exists
$$
\Theta_{p+1}^{K}\colon \bigoplus\limits_{|J|=p}K_{J}[m-p-1]
   \too\cfb_{\geq p+1}K_{\ast}\, =\,\cone\Theta_{p+2}^{K}\, ,
$$
such that the following diagram commutes:
$$
\xymatrix{
& {\bigoplus\limits_{|J|=p}K_{J}[m-p-2]}
         \ar[d]^-{(\ep_{JL}[m-p-2])_{J,L}}
            \ar[ld]_-{\Theta_{p+1}^{K}[-1]} &
\\
\cfb_{\geq p+1}[-1] \ar[r] & {\bigoplus\limits_{|L|=p+1}K_{L}[m-p-2]}
     \ar[r]^-{\Theta_{p+2}^{K}} & \cfb_{\geq p+2}K_{\ast} \rlap{\, .}
}
$$
More precisely, the morphism of complexes~$\Theta_{p+1}^{K}[p+1-m]$
is in degree~$t$ given by
$$
\xymatrix{
{\bigoplus\limits_{|J|=p}K_{J\, t}}
   \ar[rr]^-{\left(\begin{array}{c}\big((\ep_{JL})_{JL}\big)_{t}\\[1mm] 0\end{array}\right)}
      & & {\bigoplus\limits_{|L|=p+1}K_{L\, t}\,\oplus\,\cfb_{\geq p+2}[p+1-m]_{t}}\, =\,
                   {\cfb_{\geq p+1}[p+1-m]_{t}}\, .
}
$$
Therefore we have by~(\ref{matrix0Eq}) that
$\Theta_{p+1}^{K}\cdot (\ep_{IJ}[m-p-1])_{I,J}=0$
which finishes the induction step and the
definition of $\cfb_{\geq p}K_{\ast}$
for $p\geq 0$.

\smallbreak

If $p\leq 0$ we set
$\cfb_{\geq p}K_{\ast}:=\cfb_{\geq 0}K_{\ast}$.
\end{ntt}

\begin{ntt}[{\bf An exact triangle}]
\label{Cofiber-DiagSect}
By construction we have then for all $0\leq p\leq m$
a commutative diagram
{\footnotesize
$$
\xymatrix{
{\bigoplus\limits_{\begin{array}{c} |I|=p-2\\ m\not\in I\end{array}}}K_{I\cup\{ m\}}[m-p]
    \ar[r] \ar[dd]^-{\Theta^{K'}_{p-1}} &
         \;\;\; {\bigoplus\limits_{|J|=p-1}}K_{J}[m-p] \ar[r] \ar[dd]^-{\Theta^{K}_{p}} &
         {\bigoplus\limits_{\begin{array}{c} |L|=p-1\\ m\not\in I\end{array}}}K_{L}[m-p]
              \ar[dd]^-{\Theta^{\tilde{K}}_{p}}
\\
\\
\cfb_{\geq p-1}K'_{\ast} \ar[r] & \cfb_{\geq p}K_{\ast} \ar[r] &
        (\cfb_{\geq p}\tilde{K}_{\ast})[1] \rlap{\, ,}
}
$$
}
whose lower row is an exact triangle and whose upper row is a short split
exact sequence of complexes (with obvious morphisms) for all $p\geq -1$.

\smallbreak

Shifting the diagram of the lemma for $p=m-1$ to the left we get
a commutative diagram in the bounded derived category of complexes
of abelian groups:
{\footnotesize
$$
\xymatrix{
{\bigoplus\limits_{\begin{array}{c} |I|=m-2\\ m\not\in I\end{array}}}K_{I} \ar[r]^-{0}
    \ar[dd]^-{\Theta^{\tilde{K}}_{m-1}} &
\;\;\;{\bigoplus\limits_{|J|=m-2}}K_{J}[1] 
            \ar[r] \ar[dd]^-{\Theta^{K}_{m-1}} &
          {\bigoplus\limits_{\begin{array}{c} |L|=m-3\\ m\not\in I\end{array}}}K_{L\cup\{ m\}}[1]
               \ar[dd]^-{\Theta^{K'}_{m-2}}
\\
\\
K_{\{1,\ldots ,m-1\}} \ar[r]^-{\Theta} & \cfb_{\geq m-2}K'_{\ast} \ar[r] &
     \cfb_{\geq m-1}K_{\ast} \rlap{\, ,}
}
$$
}
where the arrow
{\footnotesize
$$
\Theta\colon  K_{\{ 1,\ldots ,m-1\}}\, =\,\cfb_{\geq m-1}\tilde{K}_{\ast}\too
         \qquad\qquad\qquad\qquad\qquad\qquad\qquad\qquad\qquad\qquad\qquad
$$
$$
\qquad\qquad\qquad\cfb_{\geq m-2}K'_{\ast}\, =\,\cone
  \Big(\bigoplus\limits_{\begin{array}{c} |J|=m-2\\ m\not\in J \end{array}}
   K_{J}\,\xrightarrow{\;\Theta^{K'}_{m-1}\;}K_{\{ 1,\ldots ,m\}}\Big)
$$
}
is induced by $\ep_{\{ 1,\ldots m-1\}\{ 1,\ldots m\}}\colon K_{\{ 1,\ldots m-1\}}\too
K_{\{ 1,\ldots m\}}$.
\end{ntt}

\begin{dfn}
\label{Defcofiber}
The {\it cofiber} of the
$m$-cube of complexes~$K_{\ast}$ is the complex
$$
\cfb K_{\ast}\, :=\;\cfb_{\geq 0}K_{\ast}\, .
$$
The assignment $K_{\ast}\mapsto\cfb K_{\ast}$
is a covariant functor from the category of
$m$-cubes of complexes to $\Kb (\Ab)$.
\end{dfn}

\bigbreak

\begin{ntt}[{\bf The spectral sequence}]
\label{SpSeqSubSect}
By the very definition of the complexes $\cfb_{\geq p}K_{\ast}$ we
have exact triangles
$$
\cfb_{\geq p+1}K_{\ast}\too\cfb_{\geq p}K_{\ast}\too
    \bigoplus\limits_{|I|=p}K_{I}[m-p]
$$
for any $m$-cube of complexes~$K_{\ast}$ and
all $p\geq 0$. The associated long exact homology
sequences constitute an exact couple and so we get
a convergent spectral sequence of cohomological type
$$
E^{p,q}_{1}(K_{\ast})\, :=\;
\HM_{-p-q}\big(\bigoplus\limits_{|I|=p}K_{I}[m-p]\big)\;\,
   \Longrightarrow\;\HM_{-p-q}(\cfb K_{\ast})\, .
$$
(Note that the complexes $K_{I}$ and so
also $\cfb_{\geq p}K_{\ast}$ are all
bounded.) By construction the $IJ$-component
$\HM_{-p-q}(K_{I}[m-p])\too\HM_{-p-q}(K_{J}[m-p])$ of
the differential $d_{1}^{p,q}\colon E_{1}^{p,q}(K_{\ast})
\too E_{1}^{p+1,q}(K_{\ast})$ is equal
$\HM_{-p-q}(\ep_{IJ}[m-p])$.
\end{ntt}

\begin{ex}
\label{CKconeComplExpl2}
Let $X,Z_{1},\ldots ,Z_{m}$, $K_{\ast}$,
and $\CM_{\ast}$ be as in example~\ref{CKconeComplExpl1}.
We set $W_{l}:=\bigcup\limits_{j=1}^{l}Z_{j}$ for $1\leq l\leq m$
and $\CK_{\smb}(Y):=\CK_{\smb}(Y,\CM_{n})$ for any finite
type $k$-scheme~$Y$.

The pull-back along the open immersion
$X\setminus W_{m}\hookrightarrow X$
induces a morphism of complexes
$\cfb_{\geq m}K_{\ast}=\CK_{\smb}(X)\too
\CK_{\smb}(X\setminus W_{m})$. The composition
of the morphism with $\Theta^{K}_{m}$ is zero given
a morphism of complexes $\cfb_{\geq m-1}K_{\ast}
\too\CK_{\smb}(X\setminus W_{m})$. Composing
this morphism with $\Theta^{K}_{m-1}$
is again zero and hence induce a
morphism from $\cfb_{\geq m-2}K_{\ast}$
to $\CK_{\smb}(X\setminus W_{m})$. Proceeding
further we finally get a morphism of complexes
$$
\gamma\colon \cfb K_{\ast}\too\CK_{\smb}(X\setminus W_{m})\, .
$$
Similarly we have morphisms of complexes $\tilde{\gamma}:
\cfb\tilde{K}_{\ast}\too\CK_{\smb}(Z_{m}\setminus W_{m-1})$ and
$\gamma'\colon \cfb K'_{\ast}\too\CK_{\smb}(X\setminus W_{m-1})$.
We claim that this is a quasi-isomorphism. This
is obvious for $m=1$. Let $m\geq 2$. Then by~\ref{Cofiber-DiagSect}
we have a commutative diagram whose rows are exact triangles
$$
\xymatrix{
\cfb \tilde{K}_{\ast} \ar[r] \ar[d]^-{\tilde{\gamma}} & \cfb K'_{\ast}
   \ar[r] \ar[d]^-{\gamma'} & \cfb K_{\ast} \ar[d]_-{\gamma}
\\
\CK_{\smb}(Z_{m}\setminus W_{m-1}) \ar[r]^-{\iota_{\ast}} &
   \CK_{\smb}(X\setminus W_{m-1}) \ar[r]^-{j^{\ast}}&
          \CK_{\smb}(X\setminus W_{m}) \rlap{\, ,} 
}
$$
where $\iota\colon  Z_{m}\setminus W_{m-1}\hookrightarrow
X\setminus W_{m-1}$ and $j\colon X\setminus W_{m-1}
\hookrightarrow X\setminus W_{m}$ are to
the respective subschemes corresponding open
respectively closed immersions.
The claim follows from this diagram by induction.

\smallbreak

Hence we have a convergent spectral sequence
$$
E_{1}^{p,q}(X,Z_{1},\ldots ,Z_{m},\CM_{n})\, :=\;
     \bigoplus\limits_{|I|=p}\HM_{-q-m}(Z_{I},\CM_{n})
          \;\;\;\Longrightarrow\;
            \HM_{-p-q}(U,\CM_{n})\, ,
$$
where $U=X\setminus\bigcup\limits_{i=1}^{m}Z_{i}$.
The $IJ$-component of the differential is equal~$0$
if $I\not\subset J$ and equal $(-1)^{d-1}$
times the push-forward along the closed immersion
$Z_{I}\hookrightarrow Z_{J}$ if $J=\{ i_{1}<\ldots <i_{l}\}$
and $I=J\setminus\{ i_{d}\}$.
\end{ex}

\begin{rem}
This spectral sequence applies also to Voevodsky's~\cite{CTMH}
motivic cohomology of a homotopy invariant Nisnevich sheaf
with transfers~$\sheaf{F}$. By D\'eglise's Th\`ese~\cite{Deg02}
one can associate to such a sheaf a cycle module $\hat{\sheaf{F}}_{\ast}$,
such that there is a natural isomorphism
$\HM^{i}_{\Nis}(X,\sheaf{F})\simeq\HM^{i}(X,\hat{\sheaf{F}}_{0})$
for all $i\in\N$. (Vice versa, if~$\CM_{\ast}$ is a cycle module
then $X\mapsto\HM^{i}(X,\CM_{n})$ is a homotopy invariant Nisnevich
sheaf with transfers.)
\end{rem}

\begin{ex}
\label{SingHomExpl}
Let $X$ be a topological space with closed subspaces $Z_{1},\ldots ,Z_{m}$.
As in the book~\cite{LAT} of Dold we denote by~$SX$ and~$S(X,A)$ the
singular and relative singular complex of~$X$ and the pair $(X,A)$
with $A\subseteq X$ a closed subset, respectively.
Let $\HM_{\ast}(X)$ and $\HM_{\ast}(X,A)$ be the homology
groups of these complexes, \ie the {\it (relative) singular homology} of the
space~$X$ and the pair~$(X,A)$, respectively.

We set (as above) $Z_{I}:=\bigcap\limits_{j\not\in I}Z_{j}$ for subsets
$I\subseteq\{ 1,\ldots ,m\}$, and denote by $\iota_{IJ}$ the embedding
$Z_{I}\hookrightarrow Z_{J}$ if $I\subseteq J$.

\smallbreak

The map $I\mapsto SZ_{I}$ is then a $m$-cube
of complexes and we get by the same reasoning as in
Example~\ref{CKconeComplExpl2} a convergent spectral sequence
of cohomological type
$$
E_{1}^{p,q}(X,Z_{1},\ldots ,Z_{m})\, :=\;
  \bigoplus\limits_{|I|=p}\HM_{-q-m}(Z_{I})\;\;\Longrightarrow\;
     \HM_{-p-q}(X,\bigcup\limits_{i=1}^{m}Z_{i})\, ,
$$
where the $IJ$-component of the differential $d_{1}^{p,q}\colon 
E_{1}^{p,q}\too E_{1}^{p+1,q}$ is zero if $I\not\subset J$ and
equal $(-1)^{r-1}\cdot\HM_{-q-m}(\iota_{IJ})$ if
$J=\{ i_{1},\ldots ,i_{p},i_{p+1}\}$ and $I=J\setminus\{ i_{r}\}$
for some $1\leq r\leq p+1$.
\end{ex}

\end{appendix}

%%%%%%%%%%%%%%%%%%%%%%%%%%%%%%%%%%%%%%%
%%%%%%%%%%%%%%%%%%%%%%%%%%%%%%%%%%%%%%%

\smallbreak

\paragraph{\bf Acknowledgments.}

We would like to thank Roland L\"otscher for
several helpful comments on quotients of group scheme actions.
The first author has been supported by the Deutsche 
Forschungsgemeinschaft, GI 706/1-2 and GI 706/2-1.
The second author has been supported by the NSERC Discovery 385795-2010,
Accelerator Supplement 396100-2010 and the Early Researcher Award grants.

\bibliographystyle{plain}

\end{document}